\documentclass[preprint]{elsarticle}
\RequirePackage{amsthm,amsmath,amssymb}
\newcommand{\set}[1]{\left\{#1\right\}}
\newcommand{\norm}[1]{\lVert#1\rVert}
\newcommand{\ind}[1]{\mathbf{1}_{#1}}
\newcommand{\abs}[1]{\left\vert#1\right\vert}
\newcommand{\pr}{\mathsf P}
\newcommand{\mex}[1]{\mathsf{E}\left[\,#1\,\right]}

\newcommand{\sign}{\operatorname{sign}}
\newcommand{\R}{{\mathbb R}}
\newcommand{\F}{\mathcal F}

\newcommand{\wLH}{\widehat L_H^2}

\newcommand{\wotimes}{\mathop{\widehat\otimes}}
\newcommand{\wick}{\mathop{\diamond}}
\newcommand{\pair}[1]{\left\langle #1 \right\rangle}

\numberwithin{equation}{section}

\newtheorem{theorem}{Theorem}[section]
\newtheorem{lemma}{Lemma}[section]
\theoremstyle{remark}
\newtheorem{remark}{Remark}[section]
\newtheorem{example}{Example}[section]
\theoremstyle{definition}
\newtheorem{definition}{Definition}[section]

\journal{Stochastic Processes and their Applications}
\begin{document}
\begin{frontmatter}
\title{Random variables as pathwise integrals with respect to fractional
Brownian motion\tnoteref{T1}}

\tnotetext[T1]{Yuliya Mishura thanks Aalto University School of Science for partial support of this research. Georgiy Shevchenko and Esko Valkeila were partially supported by the Academy of Finland, grant no.\ 21245.}

\author[m1]{Yuliya Mishura}
\ead{myus@univ.kiev.ua}

\author[m1]{Georgiy Shevchenko\corref{c1}}

\cortext[c1]{Corresponding author}
\ead{zhora@univ.kiev.ua}

\author[m2]{Esko Valkeila}

\address[m1]{Kiev National Taras Shevchenko University,
Department of Mechanics and Mathematics,
Volodomirska 60,
01601 Kiev,
Ukraine}

\address[m2]{Aalto University,
Department of Mathematics and Systems Analysis,
P.O. Box 11100,
FI-00076 Aalto,
Finland}

\begin{abstract}
We show that a pathwise stochastic integral with respect to fractional Brownian motion
with an adapted integrand $g$ can have any prescribed distribution, moreover, we give
both necessary and sufficient conditions when  random variables can be represented in this form.
We also prove that any random variable is a value of such integral in some improper sense.
We discuss some applications of these results, in particular, to fractional Black--Scholes model of financial market.
\end{abstract}

\begin{keyword}
fractional Brownian motion \sep
pathwise integral \sep
generalized Lebesgue--Stieltjes integral \sep
arbitrage \sep
replication\sep
divergence integral
\MSC[2010] 60G22 \sep 60H05 \sep 60G15 \sep 91G10
\end{keyword}

\end{frontmatter}

\section{Introduction}

Let $(\Omega, \mathcal F,  \mathsf P )$ be a complete  probability space endowed with a $\mathsf P$-complete left-continuous filtration $\mathbb F =
\set{\mathcal F_t,t\in [0,1]}$.
On this stochastic basis we consider a fractional Brownian motion $ B^H$ with a Hurst parameter $H >  \frac12 $, that is an $\mathbb F$-adapted  centered Gaussian
process with the covariance function
$$
\mex{B^H_t B^H_s} = \frac12 \left( s^{2H} + t^{2H} - | s-t| ^{2H} \right), \ s,t \in [0,1].
$$
Fractional Brownian motion is a popular model for long-range dependence in financial
mathematics, economics and natural sciences.
It is well known that $B^H$ has a continuous modification, and from now on we will assume that this modification is chosen. For more information on
fractional Brownian motion, see \cite{m}.

Dudley \cite{dudley} showed that any functional $\xi$ of a standard Wiener process $W=\set{W_t,t\in[0,1]}$ can be represented as an It\^o stochastic integral $\int_0^1 \psi_t dW_t$, where $\psi$ is adapted to the natural filtration of $W$ and $\int_0^1 \psi_t^2 dt<\infty$ a.s. On the other hand, under an additional assumption $\int_0^1 \mex{\psi_t^2} dt<\infty$, only centered random variables with finite variances can be represented in this form and moreover $\psi$ is unique in this representation.

In this paper we study such  questions for fractional Brownian motion. The integral we use is a (generalized) Lebesgue--Stieltjes integral, which is defined in a pathwise sense. Although the definition of the integral differs from the one for Wiener process, the results are similar in spirit to
those of \cite{dudley}. Precisely, our findings are as follows.
We prove first that for any given distribution $F$ there exists an adapted process $\zeta$ such that $\int_0^1 \zeta_t dB^H_t$
has distribution $F$. Then we show that for any $\mathcal F_1$-measurable random variable $\xi$ there exists an $\mathbb F$-adapted process $g$ such that $\lim_{t\to 1-} \int_0^t g_s dB^H_s = \xi$, i.e. $\xi$ can
be represented as the integral $\int_0^1 g_t dB^H_t$, understood in an improper sense. We moreover show that
$\xi=\int_0^1 g_t dB^H_t$ in a proper sense under additional assumption that $\xi$ is the final value of a H\"older continuous $\mathbb F$-adapted process. In addition, if $g$ is continuous, then this condition is not only sufficient, but also necessary. As a financial implication of these
results, we get that in the fractional Black--Scholes model there exists a strong arbitrage and any contingent claim may be weakly
hedged with an arbitrary hedging cost.

The paper has the following structure. In Section 2 we give some preliminaries on pathwise integration with respect to fBm. In Section 3 the main results are presented. In Section 4 we discuss applications of the results to mathematical finance and
to zero integral problem, from which this research originates.

\section{Preliminaries}

We understand the integral with respect to fractional Brownian motion in a pathwise sense
and define it as the generalized fractional Lebesgue--Stieltjes integral
(see \cite{samko,zahle}). %It is defined as follows.

Consider two  continuous functions  $f$ and $g$, defined on some interval $[a,b]\subset [0,1]$.
For $\alpha\in (0,1)$ define fractional derivatives
\begin{gather*}
\big(D_{a+}^{\alpha}f\big)(x)=\frac{1}{\Gamma(1-\alpha)}\bigg(\frac{f(x)}{(x-a)^\alpha}+\alpha
\int_{a}^x\frac{f(x)-f(u)}{(x-u)^{1+\alpha}}du\bigg)1_{(a,b)}(x),\\
\big(D_{b-}^{1-\alpha}g\big)(x)=\frac{e^{i\pi
\alpha}}{\Gamma(\alpha)}\bigg(\frac{g(x)}{(b-x)^{1-\alpha}}+(1-\alpha)
\int_{x}^b\frac{g(x)-g(u)}{(x-u)^{2-\alpha}}du\bigg)1_{(a,b)}(x).
\end{gather*}
Assume that
 $D_{a+}^{\alpha}f\in L_p[a,b], \ D_{b-}^{1-\alpha}g_{b-}\in
L_q[a,b]$ for some $p\in (1,1/\alpha)$, $q = p/(p-1)$, where $g_{b-}(x) = g(x) - g(b)$.

Under these
assumptions, the generalized Lebesgue--Stieltjes
integral %$\int_a^bf(x)dg(x)$
is defined as
\begin{equation*}\int_a^bf(x)dg(x)=e^{-i\pi\alpha}\int_a^b\big(D_{a+}^{\alpha}f\big)(x)\big(D_{b-}^{1-\alpha}g_{b-}\big)(x)dx.
\end{equation*}
It was shown in \cite{samko} that for any $\alpha\in(1-H,1)$ there exists the fractional derivative
$D_{b-}^{1-\alpha}B^H_{b-}\in L_{\infty}[a,b].$ Hence, for
$f$ with $D_{a+}^\alpha f\in L_1[a,b]$ we can define the integral w.r.t.
fBm according to this formula:
\begin{equation*}
%\label{fBm}
\int_a^bf_sdB^H_s=e^{-i\pi\alpha}\int_a^b\big(D_{a+}^{\alpha}f\big)(x)\big(D_{b-}^{1-\alpha}B^H_{b-}\big)(x)dx.
\end{equation*}

In view of this, we will consider the following norm for $\alpha \in(1-H,1/2)$:
\begin{gather*}
\norm{f}_{1,\alpha,[a,b]} = \int_a^b \left(\frac{\abs{f(s)}}{(s-a)^\alpha} + \int_a^s \frac{\abs{f(s)-f(z)}}{(s-z)^{1+\alpha}}dz\right)ds.
\end{gather*}
For simplicity we will abbreviate $\norm{\cdot}_{\alpha,t} = \norm{\cdot}_{1,\alpha,[0,t]}$. The following estimate for
$t\le 1$ is clear:
\begin{gather*}
\abs{\int_0^t f(s) dB^H_s}\le K_\alpha(\omega) \norm{f}_{\alpha,t},
\end{gather*}
where $K_\alpha(\omega) = \sup_{0\le u<s\le 1} \abs{D_{u-}^{1-\alpha}B^H_{s-}}<\infty$ a.s.

We will need the following version of It\^o formula for fractional Brownian motion. It was proved in \cite{amv} for convex functions $F$, but a careful
analysis of the proof shows that the following result is true.
\begin{theorem}\label{ito}
Let $f:\R\to \R$ be a function of locally bounded variation, $F(x) = \int_0^x f(y) dy$. Then for any $\alpha\in(1-H,1/2)$
\ $\norm{f(B^H_\cdot)}_{\alpha,1}<\infty$ a.s. and
$$
F(B^H_t) = \int_0^t f(B^H_s) dB^H_s.
$$
\end{theorem}

Throughout the paper all unimportant constants will be denoted by $C$, and their value may change from line to line. Random
constants will be denoted by $C(\omega)$.

\section{Main results}

\subsection{Auxiliary construction}\label{s:main}

%boundedness???
In this section we construct the essential ingredient of some results: an adapted
process such that w.r.t.\ fractional Brownian motion it integrates (in improper sense) to infinity.

The key fact is the following well-known small ball estimate for fractional Brownian motion (see e.g.\ \cite{lishao}):
there is a
constant $c>0$, independent of $\epsilon $ and $T$, such that
\begin{equation}\label{eq:m-r}
P \left( \sup _{t\in [0,T ] } | B^H _t | < \epsilon \right)
\le e^{-cT \epsilon ^{-1/H}} \ \ \mbox{for } \epsilon \le T^ {H}.
\end{equation}

\begin{lemma}\label{mainlemma}
There exists an $\mathbb F$-adapted process $\varphi=\set{\varphi_t,t\in[0,1]}$ such that
\begin{itemize}
\item For any $t<1$ and $\alpha\in(1-H,1/2)$ \ $\norm{\varphi}_{\alpha,t}<\infty$ a.s., so integral $v_t = \int_0^t \varphi _s dB^H_s $ exists as a generalized Lebesgue--Stieltjes integral.
\item $\lim _{t\to 1-} v_t = \infty $ a.s.
\end{itemize}
\end{lemma}
\begin{proof}
Fix arbitrary  $\gamma\in(1,1/H)$ and $\beta \in(0,\frac{1}{\gamma H} -1)$.
Denote $\Delta_n = n^{-\gamma}/\zeta(\gamma)$, $\zeta(\gamma) = \sum_{n\ge 1}n^{-\gamma}$,
and define $t_0=0$, $t_n = \sum_{k=1}^{n} \Delta_k$, $n\ge 1$, so that $t_n\to 1-$, $n\to\infty$.
Denote also $f_\beta(x) = (1+\beta) x^\beta \sign x$, so that $\int_0^x f_\beta(z) dz = \abs{x}^{1+\beta}$, $x\in\mathbb R$.

Let $\tau_n = \min\set{t\ge t_{n-1}: \abs{B_t^H - B_{t_{n-1}}^H}\ge n^{-1/(1+\beta)}}\wedge t_{n}$ and define
$$
\varphi_t = \sum_{n=1}^\infty f_\beta(B^H_t - B^H_{t_n})\ind{[t_{n-1},\tau_{n})}(t).
$$
%$\Delta B^H_{n} = B^H_{t_{n}} - B^H_{t_{n-1}}$, $n\ge 1$.
%$f_\beta(x) = (1+a) x^a\ind{[0,1]}(x)$, $F_\beta(x) = \int_0^x f_\beta(y)dy = (0\vee x)^{1+a}\wedge 1$,
First we establish estimate $\norm{\varphi}_{\alpha,t}<\infty$ a.s. To that end, note that fractional Brownian
motion $B^H$ is almost surely bounded on $[0,1]$ and write
$
{\norm{\varphi}_{\alpha,t_n}} = I_1 + I_2,
$
where
\begin{gather*}
I_1 = \int_0^{t_n} {{\abs{\varphi_s}}{s^{-\alpha}}}ds \le C(\omega),\\
I_2 = \int_0^t \int_0^s {{\abs{\varphi_s-\varphi_u}}{(s-u)^{-1-\alpha}}}du\,ds\\
 = \sum_{k=1}^{n} \int_{t_{k-1}}^{t_{k}} \left( \int_0^{t_{k-1}} +\int_{t_{k-1}}^s \right){{\abs{\varphi_s-\varphi_u}}{(s-u)^{-1-\alpha}}}du\,ds.
\end{gather*}
Now estimate
\begin{gather*}
\sum_{k=1}^{n} \int_{t_{k-1}}^{t_{k}} \int_0^{t_{k-1}}{{\abs{\varphi_s-\varphi_u}}{(s-u)^{-1-\alpha}}}du\,ds\\
\le C(\omega)\sum_{k=1}^{n} \int_{t_{k-1}}^{t_{k}} \int_0^{t_{k-1}}{(s-u)^{-1-\alpha}}du\,ds\\
\le C(\omega) \sum_{k=1}^{n} \int_{t_{k-1}}^{t_{k}} {(t-t_{k-1})^{-\alpha}}ds \le C(\omega)\sum_{k=1}^{n} \Delta_k^{1-\alpha} <\infty.
\end{gather*}
Finally,
\begin{gather*}
\int_{t_{k-1}}^{t_{k}} \int_{t_{k-1}}^s {\abs{\varphi_s-\varphi_u}}{(s-u)^{-1-\alpha}}du\,ds\\
=  \int_{\tau_{k}}^{t_{k}} \int_0^{\tau_{k}} \abs{f_\beta(B^H_u - B_{t_k}^H)}(s-u)^{-1-\alpha}du\,ds + I_k\\
%+\int_{t_k}^{\tau_k} \int_{t_k}^{s} {{\abs{f_\beta(B^H_s - B_{t_k}^H)-f_\beta(B^H_u - B_{t_k}^H)}}{(s-u)^{-1-\alpha}}}du\,ds\\
\le C(\omega) \int_{\tau_{k-1}}^{t_{k}} (s-\tau_k)^{-\alpha} ds + I_k\le C(\omega) + I_k,
\end{gather*}
where
$$
I_k = \int_{t_{k-1}}^{\tau_k} \int_{t_{k-1}}^{s} {{\abs{f_\beta(B^H_s - B_{t_k}^H)-f_\beta(B^H_u - B_{t_k}^H)}}{(s-u)^{-1-\alpha}}}du\,ds
$$
is finite almost surely by Theorem \ref{ito}.

Now by the It\^o formula, for $t\in[t_{n-1},t_{n})$
$$
v_t = \int _0^t \varphi _s dB^H_s = \sum_{k=1}^{n-1} \abs{\Delta B^H_k}^{1+\beta} + \abs{B^H_{t\wedge \tau_n} - B^H_{t_{n-1}}}^{1+\beta},
$$
where $\Delta B^H_{k} = B^H_{\tau_{k}} - B^H_{t_{k-1}}$, $k\ge 1$. It is easy to see that $v_t\ge v_{t_n}$ for $t\ge t_n$, so in order to prove that $v_t\to \infty$, $t\to 1$, it is enough to show that $v_{t_n}\to \infty$, $n\to \infty$, which in turn is equivalent to
$\sum_{n=1}^\infty\abs{\Delta B^H_n}^{1+\beta}=\infty$. Observe that $\abs{\Delta B^H_n}^{1+\beta}\ge 1/n$ provided that
$\tau_n<t_{n}$. Therefore, defining $A_n = \set{\sup_{t\in [t_{n-1},t_{n}]} \abs{B^H_t -B^H_{t_{n-1}}}< n^{-1/(1+\beta)}}$, $n\ge 1$,
it is enough to show that almost surely only finite number of the events $A_n$ happens.
Using the small ball estimate \eqref{eq:m-r} and stationarity of increments of $B^H$, we obtain
\begin{gather*}
\pr(A_n) = \pr\left( \sup_{t\in [0,\Delta_n]} \abs{B^H_t}< n^{-1/(1+\beta)} \right) \le \exp\set{-c \zeta(\gamma)^{-1} n^{-\gamma+\frac{1}{H(1+\beta)}}},
\end{gather*}
so $\sum_{n\ge 1} P(A_n)<\infty$ since $\frac{1}{H(1+\beta)}>\gamma$. Thus, we get the desired statement from the Borel-Cantelli lemma.
\end{proof}
\begin{remark}
It is easy to see  that $\norm{\varphi}_{\alpha,t}<\infty$ even for a random $t=t(\omega)<1$.
\end{remark}

\subsection{Stochastic integral with respect to fBm can have any distribution}
The following result is about representation, not of a random variable, but rather of a distribution. From the financial point of view, it means that an investor can get any desired risk profile, using a self-financing portfolio (see Theorem~\ref{distrfmthm}). The key for its proof is Lemma~\ref{mainlemma}, the rest of the proof goes exactly as in \cite{dudley}.
\begin{theorem}\label{distrthm}
For any distribution function $F$ there exists an adapted process $\zeta$ such that
$\norm{\zeta}_{\alpha,1}<\infty$ and the distribution function of $\int_0^1 \zeta_s dB^H_s$ is $F$.
\end{theorem}
\begin{proof}
It is clear that there exists a non-decreasing function $g:\R\to\R$ such that $g(B^H_{1/2})$ has distribution $F$. So it is enough to construct
an adapted process  $\zeta$  such that $\int_0^1 \zeta_s dB^H_s = g(B^H_{1/2})$. Let $\varphi$ be the process constructed in Lemma~\ref{mainlemma}, $v_t = \int_{1/2}^t \varphi_s dB^H_s$. Define $\tau = \min\set{t\ge 1/2:v_t = \abs{g(B^H_{1/2})}}$. Since $v_t\to\infty$ as $t\to 1-$ a.s.,
we have $\tau<1$ a.s. Now put $$\zeta_t = \varphi_t\sign g(B^H_{1/2}) \mathbf 1_{[1/2,\tau]}(t).$$
We have
\begin{gather*}
\norm{\zeta}_{\alpha,1} = \norm{\varphi}_{\alpha,\tau} + \int_\tau^1\int_0^\tau \frac{\abs{\zeta_s}}{(t-s)^{1+\alpha}}ds\,dt\le C(\omega) +
C(\omega)\int_\tau^1 (t-\tau)^{-\alpha}dt <\infty.
\end{gather*}
Obviously,
$$
\int_0^1 \zeta_s dB^H_s = g(B^H_{1/2}),
$$
whence the statement follows.
\end{proof}
\begin{remark}
Observe that the process $\zeta$ is adapted not only to the filtration $\mathbb F$, but also to the natural filtration of $B^H$; moreover, this process is piecewise H\"older continuous, which implies that the integral is not only well-defined in the generalized Lebesgue--Stieltjes sense, but it is also a limit of integral sums.
\end{remark}

\subsection{Any random variable is a stochastic integral in improper sense}
\begin{theorem}\label{thmimproper}
For any $\mathcal F_1$-measurable variable $\xi$ there exists an $\mathbb F$-adapted process $\psi$ such that
\begin{itemize}
\item For any $t<1$ and $\alpha\in(1-H,1/2)$ \ $\norm{\psi}_{\alpha,t}<\infty$ a.s.
\item $\lim_{t\to 1 } \int_0^t \psi_s dB^H_s = \xi $ a.s.
\end{itemize}
\end{theorem}
\begin{proof}
Process $z_t = \tan \mathbb E [\arctan \xi\vert \mathcal F_{t}]$
is  $\mathbb F$-adapted and $z_t\to \xi$, $t\to 1-$, a.s.\ by the left continuity of $\mathbb F$ and the martingale convergence.

Let $\{t_n,n\ge 1\}$ be arbitrary increasing sequence of points from $[0,1]$ converging to $1$.

By Lemma~\ref{mainlemma}, there exists  an $\mathbb F$-adapted process $\varphi^n$ on $[t_n,t_{n+1}]$ such that $v^n_t = \int_{t_n}^t
\varphi_s^n dB^H_s \to +\infty$, $t\to t_{n+1}-$.

Now denote $\xi_n = z_{t_n}$ and $\delta_n = \xi_{n} - \xi_{n-1}$, $n\ge 2$, $\delta_1 = \xi_1$. Take $\tau_n = \min\set{t\ge t_n: v^n_{t} = \abs{\delta_n}}$ and define
$$
\psi_t = \sum_{n\ge 1} \varphi_t^n \ind{[t_{n},\tau_n]}(t)\sign \delta_n,\quad x_t = \int_0^t \psi_s dB^H_s.
$$
The finiteness of norm $\norm{\psi}_{\alpha,t}$ for $t<1$ is proved as in Lemma \ref{mainlemma} and Theorem \ref{distrthm}.
It is clear that $x_{t_{n+1}} = \sum_{k=1}^n \delta_n = \xi_n$, so $x_{t_{n+1}}\to \xi$, $n\to\infty$. Moreover, from the construction of process $\psi$ it follows that for $t\in[t_n, t_{n+1}]$ the value $x_t$ is between $\xi_{n-1}$ and $\xi_n$, whence $x_t\to \xi, t\to 1-$.
\end{proof}

\subsection{Which variables can be represented as stochastic integrals}
For some random variables $\xi$ we can claim even more: the existence of an $\mathbb F$-adapted $g_t$ such that $\int_0^1 g_s dB^H_s$ is well-defined and is equal to $\xi$.
To establish the main result here, we need an auxiliary lemma.
\begin{lemma}\label{probfbm<0}
For $0< s\le t\le 1$ \ \begin{equation}
\label{probest}
\mathsf P(B^H_sB^H_t\le 0)\le C(t-s)^H t^{-H}.
\end{equation}
\end{lemma}
\begin{proof} As the distribution of $B^H$ is symmetric and continuous, it is enough to estimate $\mathsf P(B^H_s< 0< B^H_t)$.
By the self-similarity of $B^H$,
$$
\mathsf P(B^H_s< 0< B^H_t) = \mathsf P(B^H_{s/t}< 0< B^H_1).
$$
If $s/t$ is small (less than $1/2$ say), then $\abs{t-s}^H t^{-H} = \abs{1-s/t}^{H}> 2^{-H}$, so \eqref{probest} holds
with $C=2^{H}$. Thus, we only have to consider $s/t$ close to $1$. Denote $u=s/t$, $\rho(u) = u^{-H}\mex{B^H_u B^H_1}
= (u^{-H}+u^{H} - u^{-H}(1-u)^{2H})/2$. Write $$1-\rho(u) = u^{-H} \left(
(1-u)^{2H} - (1-u^H)^2\right)\le C(1-u)^{2H},$$
so we can estimate
\begin{gather*}
\mathsf P(B^H_{u}< 0< B^H_1) = \frac{1}{2\pi (1-\rho(u)^2)^{1/2}}\int_0^\infty \int_{-\infty}^0 \exp\set{-\frac{x^2 + y^2 - 2\rho(u) xy}{2(1-\rho(u)^2)}}dx\,dy\\
= \frac{(1-\rho(u)^2)^{1/2}}{2\pi }\int_0^\infty \int_{-\infty}^0 \exp\set{-(x^2 + y^2 - 2\rho(u) xy)/2}dx\,dy\\
\le C(1-\rho(u))^{1/2} \int_0^\infty \int_{-\infty}^0 \exp\set{-(x-y)^2/2}dx\,dy
 \le  C(1-u)^H.
\end{gather*}
A simple observation that $(1-u)^H = t^{-H}(t-s)^H$ concludes the proof.
\end{proof}
\begin{theorem}\label{abreakthroughtheorem}
Let for a random variable $\xi$ there exist a number $a>0$ and an $\mathbb F$-adapted almost surely
$a$-H\"older continuous process $\set{z_t,t\in[0,1]}$
such that $z_1 = \xi$. Then for any $\alpha \in (1-H, (1-H+a)\wedge 1/2)$ there exists an $\mathbb F$-adapted
process $\psi$ such that $\norm{\psi}_{\alpha,1}<\infty$ and $\int_0^1
\psi_s dB^H_s = \xi$.
\end{theorem}
\begin{proof}
We can assume without loss of generality that $a<H$.

\emph{Step 1. Construction.}
Take some $\gamma >(1-\alpha-H+a)^{-1}>1$ and put $\Delta_n = n^{-\gamma}/\zeta(\gamma)$,  $t_0=0$, $t_n = \sum_{k=1}^{n} \Delta_k$, $n\ge 1$.
For brevity, denote $\xi_n = z_{t_n}$, $\delta_n = \abs{\xi_n-\xi_{n-1}}$.

We construct process $\psi$ recursively on intervals $(t_n,t_{n+1}]$. First, set $\psi_t=0$ for $t\in[t_0,t_1]$ and fix some $\kappa\in\big(\gamma(H-a),\gamma(1-\alpha)-1\big)$.
This is possible because $\gamma(H-a)<\gamma(1-\alpha)-1$ thanks to
the choice of $\gamma$. Observe also that $a>H-\kappa/\gamma$ with such choice of $\kappa$. %Also fix some $b\in (H-\kappa/\gamma,a)$.

Denote $v_t = \int_0^t \psi_s dB^H_s$.  If $\psi$ is constructed on $[t_0,t_{n-1}]$ for some $n\ge 2$, we will show
how to construct it on $(t_{n-1},t_n]$.  To this end, consider two cases.

\underline{Case A} $v_{t_{n-1}}=\xi_{n-2}$. Define $$\tau_n = \min\set{t\ge t_{n-1}: n^{\kappa}\abs{B^H_t - B^H_{t_{n-1}}} =
\delta_{n-1}}\wedge t_{n}$$ and set
$$\psi_t = n^\kappa \sign(B^H_t - B^H_{t_{n-1}})\sign(\xi_{n-1}-\xi_{n-2})\ind{t\le \tau_n}$$ for $t\in[t_{n-1},t_{n})$.
By the It\^o formula,
\begin{equation*}%\label{vtn}
v_{t_{n}} = v_{t_{n-1}} + n^{\kappa}\abs{B^H_{\tau_n}-B^H_{t_{n-1}}}\sign(\xi_{n-1}-v_{t_{n-1}}),
\end{equation*}
so  we have $v_{t_n} = \xi_{n-1}$ provided $\tau_n < t_n$.

\underline{Case B} $v_{t_{n-1}}\neq \xi_{n-2}$. In this case we use a construction similar to that of Theorem~\ref{thmimproper}. Namely,
let $\varphi_t^n$ be an adapted process on $[t_{n-1},t_n]$ such that $v_t^n := \int_{t_{n-1}}^{t_n} \varphi_s^n dB^H_s \to \infty$,
$t\to t_n-$, define $\tau_n =  \min\set{t\ge t_{n-1}: v_t^n =
\abs{\xi_{n-1}-v_{t_{n-1}}}}$ and set $\psi_t = \varphi_t^n \sign(\xi_{n-1}-v_{t_{n-1}})\ind{t\le \tau_n}$ for $t\in[t_{n-1},t_{n})$.
Then $v_{t_{n}} = \xi_{n-1}$.

\emph{Step 2.} We argue that almost surely there is $N(\omega)$ such that $v_{t_{n}}=\xi_{n-1}$ for $n\ge N(\omega)$.
As in the proof of Lemma~\ref{mainlemma}, we will use Borel--Cantelli lemma and the small ball estimate \eqref{eq:m-r}.
For brevity, we will omit the
phrase ``almost surely'' in the rest of the proof.

Define events $$C_n = \set{\sup_{t\in [t_{n-1},t_{n}]} n^{\kappa}\abs{B^H_t -B^H_{t_{n-1}}}\le \delta_{n-1}},\ n\ge 2.$$
We are going to show
that only finite number of $C_n$ happens. Take some $b\in (H-\kappa/\gamma,a)$. By our assumption,
$$\delta_{n-1} = \abs{z_{t_{n-1}} - z_{t_{n-2}}}\le C(\omega) \Delta_{n-1}^a \le C(\omega) \Delta_{n}^a.$$
There exists $N_1(\omega)$ such that $C(\omega) \Delta_n^a\le \Delta_n^b$ for $n\ge N_1(\omega)$,
therefore
\begin{equation}
\label{N1omega}
\delta_{n-1}\le \Delta_n^b \quad\text{ for }n\ge N_1(\omega).
\end{equation}
So it is enough to prove that only finite number of events
$$
D_n = \set{\sup_{t\in [t_{n-1},t_{n}]} n^{\kappa}\abs{B^H_t -B^H_{t_{n-1}}}\le \Delta_n^b}.
$$
happens.
The increments of fBm $B^H$ are stationary, hence by the small ball estimate for
$n$ sufficiently large
\begin{gather*}
\mathsf P(D_n) = \mathsf P\left(\sup_{t\in [0,\Delta_n]} n^\kappa \abs{B^H_t}< \Delta_n^{b}\right)\\
= \mathsf P\left(\sup_{t\in [0,\Delta_n]}  \abs{B^H_t}< \zeta(\gamma)^{-b}n^{-b\gamma-\kappa}\right)\le \exp\set{- c\zeta(\gamma)^{{b}/{H}-1}
n^{-\gamma + (b\gamma + \kappa)/H}}.
\end{gather*}
Since $b>H-\kappa/\gamma$, equivalently, $-\gamma + (b\gamma + \kappa)/H>0$, we have that $\sum_{n=1}^\infty \mathsf P(D_n)<\infty$.
So by the Borel--Cantelli lemma, only finite number of events $D_n$ happens. As we have already noted, this implies that the same is true for $C_n$.
Thus, for some $N(\omega)$ we have $\sup_{t\in [t_{n-1},t_{n}]} n^{\kappa}\abs{B^H_t -B^H_{t_{n-1}}}> \delta_{n-1}$ for all $n\ge N(\omega)$.
This implies that $v_{t_{M}}=\xi_{M-1}$ no matter whether we have Case A or B on $(t_{M-1},t_M]$, moreover, we have Case A on $(t_{n-1},t_n]$
and $v_{t_n} = \xi_{n-1}$ for all $n\ge N(\Omega)+1$.

\emph{Step 3.} Now we prove that $\norm{\psi}_{\alpha,1}<\infty$ a.s.

Let for $n\ge 2$ $$A_n = \set{\text{We have Case A on the interval} (t_{n-1},t_{n}]},\quad B_n = A_n^c.$$
Write $\psi_t = \psi_t^A + \psi_t^B$, where
$$
\psi_t^A = \psi_t\sum_{n= 2}^\infty  \ind{(t_{n-1},t_{n}]}(t) \ind{A_n}
$$
and $\psi_t^B$ is defined similarly. By Step 2, only finite numbers of the events $B_n$ take place, hence 
the finiteness of $\norm{\psi^B}_{\alpha,1} $ can be proved  as in Lemma \ref{mainlemma} and Theorem \ref{distrthm}.

It remains to prove that $\norm{\psi^A}_{\alpha,1}<\infty$. Clearly, it is enough to show that
$\mex{\norm{\psi^A}_{\alpha,1}}<\infty$. Write
$$
\mex{\norm{\psi^A}_{\alpha,1}} = I_1 + I_2,
$$
where
\begin{gather*}
I_1 = \int_0^1 \mex{{\abs{\psi_t}}{t^{-\alpha}}}dt \le C\sum_{n=2}^\infty \int_{t_{n-1}}^{t_{n}} n^{\kappa} dt \\
= \sum_{n=2}^\infty n^{\kappa}\Delta_{n} \le  C  \sum_{n=1}^\infty n^{\kappa-\gamma}<\infty,\\
I_2 = \int_0^1 \int_0^t \mex{{\abs{\psi_t-\psi_s}}{(t-s)^{-1-\alpha}}}ds\,dt\\
 = \sum_{n=2}^\infty \int_{t_{n-1}}^{t_{n}} \left( \int_0^{t_{n-1}} +\int_{t_{n-1}}^t \right)\mex{{\abs{\psi_t-\psi_s}}{(t-s)^{-1-\alpha}}}ds\,dt.
\end{gather*}
Now estimate the terms individually, denoting $\sigma_n(t) = \sign(B^H_t-B^H_{t_{n-1}})$:
\begin{gather*}
\sum_{n=2}^\infty \int_{t_{n-1}}^{t_{n}} \int_0^{t_{n-1}}\mex{{\abs{\psi_t-\psi_s}}{(t-s)^{-1-\alpha}}}ds\,dt\\
\le \sum_{n=2}^\infty 2n^{\kappa}\int_{t_{n-1}}^{t_{n}} \int_0^{t_{n-1}}{(t-s)^{-1-\alpha}}ds\,dt\\
\le C \sum_{n=2}^\infty  n^{\kappa}\int_{t_{n-1}}^{t_{n}} {(t-t_{n-1})^{-\alpha}}dt \le 
C\sum_{n=2}^\infty  n^{\kappa}\Delta_{n}^{1-\alpha} \le C\sum_{n=1}^\infty  n^{\kappa-\gamma(1-\alpha)}<\infty;\\
I_2':=\sum_{n=2}^\infty \int_{t_{n-1}}^{t_{n}} \int_{t_{n-1}}^t \mex{{\abs{\psi_t-\psi_s}}{(t-s)^{-1-\alpha}}}ds\,dt\\
=\sum_{n=2}^\infty n^\kappa\int_{t_{n-1}}^{t_{n}} \int_{t_{n-1}}^t \mex{\abs{\sigma_n(t)\ind{t\le \tau_n}-\sigma_n(s)\ind{s\le \tau_n}}(t-s)^{-1-\alpha}\ind{A_n}}ds\,dt
\\
\le \sum_{n=2}^\infty n^{\kappa} \int_{t_{n-1}}^{t_n} \int_{t_{n-1}}^t{\mex{\abs{\sigma_n(t)-\sigma_n(s)}+\ind{s\le \tau_n<t}}}{(t-s)^{-1-\alpha}}ds\,dt\\
\le \sum_{n=2}^\infty n^{\kappa} \int_{t_{n-1}}^{t_n} \int_{t_{n-1}}^t\frac{P\big((B^H_s-B^H_{t_{n-1}})(B^H_t-B^H_{t_{n-1}})\le 0\big)+
\mex{\ind{s\le \tau_n<t}}}{(t-s)^{1+\alpha}}ds\,dt.
%\le C\sum_{n=1}^\infty n^{\kappa}\left(\int_{t_n}^{t_{n+1}} (t-t_n)^{-H}\int_{t_n}^t {(t-s)^{H-1-\alpha}}ds\,dt
%+  \mex{(t_{n+1}-\tau_n)^{1-\alpha}}\right)
%\\
%\le C\sum_{n=1}^\infty n^{\kappa}\left(\int_{t_n}^{t_{n+1}} (t-t_n)^{-H} {(t-t_n)^{H-\alpha}}dt + n^{-\gamma(1-\alpha)}\right)\\
%\le C\sum_{n=1}^\infty  n^{\kappa-\gamma(1-\alpha)}<\infty,
\end{gather*}
By the stationarity of fBm increments  and Lemma \ref{probfbm<0}, 
$$P\big((B^H_s-B^H_{t_{n-1}})(B^H_t-B^H_{t_{n-1}})\le 0\big)\le C(t-t_{n-1})^{-H} (t-s)^{H}.
$$
Also observe that 
\begin{gather*}
\int_{t_{n-1}}^{t_n} \int_{t_{n-1}}^t{\ind{s\le \tau_n<t}}{(t-s)^{-1-\alpha}}ds\,dt \le
C\int_{\tau_n}^{t_n} (t-\tau_n)^{-\alpha}\\ \le C (t_n-\tau_n)^{1-\alpha}\le C\Delta_n^{1-\alpha}.
\end{gather*}
Then we can continue estimation:
\begin{gather*}
I_2' \le C\sum_{n=2}^\infty n^\kappa\left(\int_{t_{n-1}}^{t_n} (t-t_{n-1})^{-H} \int_{t_{n-1}}^t (t-s)^{H-1-\alpha} ds\,dt+\Delta_n^{1-\alpha}\right)\\
\le C\sum_{n=2}^\infty n^\kappa\left(\int_{t_{n-1}}^{t_n} (t-t_{n-1})^{-\alpha} ds+n^{-\gamma(1-\alpha)}\right)\le C
\sum_{n=2}^\infty n^{\kappa-\gamma(1-\alpha)}<\infty.
\end{gather*}
Concluding, $\norm{\psi}_{\alpha,1}<\infty$, as required.
\end{proof}
\begin{remark}
It is easy to see that the assumption of Theorem~\ref{abreakthroughtheorem} is equivalent to the following one:
there exists a number $a>0$, an increasing sequence $\set{t_n,n\ge 1}$ of points converging to $1$ and a sequence of random variables $\set{\xi_n,n\ge 1}$ such that $\xi_n$ is $\F_{t_n}$-measurable for any $n\ge 1$ and
\begin{equation}\label{equivformulation}
\abs{\xi_n-\xi} = O(\abs{t_n-1}^a)
\end{equation}
 a.s.\ as
$n\to\infty$. (Clearly, this condition is implied by the assumption of Theorem~\ref{abreakthroughtheorem}; vice versa it can be proved by a clever linear interpolation.)
\end{remark}

A natural question is what random variables satisfy the assumption of Theorem~\ref{abreakthroughtheorem}. Below we give some examples of such random variables.

\begin{example}$\xi = F(B^H_{s_1},\dots,B^H_{s_n})$,
where $F:\R^n\to \R$ is locally H\"older continuous with respect to each variable. In this case we can set $z_t = F(B^H_{s_1\wedge t},\dots,B^H_{s_n\wedge t})$, which is clearly H\"older continuous.
\end{example}
\begin{example}
$\xi = G(\set{B^H_s,s\in[0,1]})$, where $G\colon C[0,1]\to \mathbb R$ is locally H\"older continuous with respect to the  supremum norm on $C[0,1]$. In the case one can set $z_t = G(\set{B^H_{s\wedge t},s\in[0,1]})$.
\end{example}
\begin{example}
$\xi = \ind{A}$, $A\in\F$. Indeed, for some increasing sequence $\set{t_n,n\ge 1}$ of points converging to $1$, in view of the right continuity of $\mathbb F$, the set $A$ can be approximated by some $\F_{t_n}$-measurable sets $A_n$ in probability. Hence, some subsequence of the indicator
functions $\ind{A_n}$ (without loss of generality, the sequence itself) converges almost surely. But then $\abs{\ind{A_n}-\ind{A}}=0$ a.s.\ for all $n$ large enough, so \eqref{equivformulation} is obvious.

Consequently, any simple $\F$-measurable function also satisfies the assumption of Theorem~\ref{abreakthroughtheorem}.
\end{example}

In view of financial applications, the three examples given above and their transformations are enough, because they cover virtually all possible derivative
securities: European options, Asian options, barrier options, lookback options,  digital options etc. %Still, it is an open question whether there is an equivalent condition which is easier to check.

%Moreover, it is an open problem whether there exists a single random variable which \emph{does not} satisfy the assumption of Theorem~\ref{abreakthroughtheorem}. This question is interesting on its own, so we pose it separately (in a less problem-specific form) below.
%\begin{opn}
%Let $\set{\eta_n,n\ge 1}$ be a sequence of continuous random variables (e.g.\ independent standard Gaussian), $\set{h_n,n\ge 1}$ be a sequence of positive numbers converging to zero (e.g.\ $h_n=n^{-a}$, $a>0$). Denote $\mathcal{G}_n = \sigma\{\eta_1,\eta_2\dots,\eta_n\}$, $\mathcal{G} = \sigma\set{\eta_n,n\ge 1}$.
%
%Does there exist a $\mathcal{G}$-measurable random variable $\eta$ such that $$\limsup_{n\to \infty} h_n^{-1}\abs{\eta_n-\eta}=\infty$$
%with positive probability?
%
%Observe that the answer is positive for a sequence of independent standard Bernoulli variables. Indeed, such sequence can be understood as a binary expansion of a real number from $[0,1]$. Thus, the question is about the rate of approximation of a Borel function by simple functions which are constant on the intervals of a dyadic partition, and it is easy to construct a function, which cannot be approximated with a power rate.
%\end{opn}

Further we will show that the assumption of Theorem \ref{abreakthroughtheorem} is not only natural, but also is close to be a criterion: it is a necessary condition under additional assumption that $\psi$ is \emph{continuous}.
\begin{theorem}
Let $\xi$ be an $\mathcal F_1$-measurable random variable and let there
exist an $\mathbb F$-adapted continuous process $\psi$ such that for some $\alpha>1-H$ \ $\norm{\psi}_{\alpha,1}<\infty$ a.s. and $\int_0^1 \psi_s dB^H_s = \xi$.
Then the assumption of Theorem \ref{abreakthroughtheorem} is satisfied.
\end{theorem}
\begin{proof}
Thanks to the Garsia--Rodemich--Rumsey inequality \cite{GRR}, it follows from continuity of $\psi$  and estimate $\norm{\psi}_{\alpha,1} <\infty$   that $\psi$ is almost surely H\"older continuous of any order $a<\alpha$. We also know that $B^H$ is almost surely H\"older continuous of any order $b<H$. Then by the well known property of the Lebesgue--Stieltjes integral (which is Young integral in this situation),  $z_t = \int_0^t \psi_s dB^H_s$ is almost surely H\"older continuous of any order $c<\alpha$.
 $\mathbb F$-adaptedness of $z$ is obvious.
\end{proof}

For completeness, we give the following example showing that there exist random variables which do not satisfy the assumption of Theorem~\ref{abreakthroughtheorem} even in the case where the filtration $\mathbb F$ is generated by $B^H$.

\begin{example}
Assume that $\mathbb F = \set{\mathcal F_t=\sigma(B_s^H,s\in[0,t]),t\in [0,1]}$. It is well known (see \cite{nvv}) that there
exists a Wiener process $W$ such that its natural filtration coincides with $\mathbb F$. Define
$\xi = \int_{1/2}^1 g(t){dW_t},$ where $g(t) = (1-t)^{-1/2}\abs{\log (1-t)}^{-1}$.
We will show that $\xi$ does not satisfy the assumption of Theorem~\ref{abreakthroughtheorem}. Roughly, the idea is that the best, at least in the mean-square sense, $\mathbb F$-adapted approximation of $\xi$ is $z_t = \int_0^t g(t) dW_t$, but $z_t$ is not H\"older continuous at $1$.

Without loss of generality we assume
that $W$ is defined on the classical Wiener space, i.e.\ $\Omega = C[0,1]=\set{\omega(t),t\in[0,1]}$, $W_t(\omega) = \omega(t)$, $\pr$ is the Wiener
measure, $\F_t$ is the $P$-completion of the $\sigma$-algebra generated by events of the form $\set{\omega(u)\in B}$, $u\le t$, $B\in\mathcal B(\mathbb R)$. Arguing by contradiction, put $t_n = 1-1/(n+1)$, $n\ge 1$ and let $\set{\zeta_n,n\ge 1}$ be a sequence of random variables
such that $\zeta_n$ is $\F_{t_n}$-measurable for each $n\ge 1$, and for some $a>0$
$$
%K(\omega):=
\sup_{n\ge 1} n^a\abs{\xi(\omega)-\zeta_n(\omega)}<\infty \quad\text{a.s.}
$$
Decompose $\xi=\xi_n+\eta_n $, where $\xi_n = \int_0^{t_n} f(t) dW_t$ is $\F_{t_n}$-measurable, $\eta_n = \int_{t_n}^1 f(t) dW_t$ is independent of $\F_{t_n}$. Then we have
\begin{equation}\label{alphan}
%K(\omega)=
\sup_{n\ge 1} n^a\abs{\eta_n(\omega)+\alpha_n(\omega)}<\infty \quad\text{a.s.},
\end{equation}
where $\alpha_n(\omega) = \xi_n-\zeta_n$ is $\F_{t_n}$-measurable.  Define the following bijective transformation on $\Omega$:
$$
\psi_n(\omega)(t) = \begin{cases}
\omega(t), &t\in [0,t_n],\\
2\omega(t_n) -\omega(t), &t\in(t_n,1]
\end{cases}
$$
(we reflect the path after the point $t_n$). It is clear that $\psi_n$ is measurable and preserves the measure $P$. In particular,
$%K(\psi_n(\omega))=
\sup_{n\ge 1} n^a\abs{\eta_n(\psi_n(\omega))+\alpha_n(\psi_n(\omega))}<\infty$ a.s. It easy to check that
$\eta_n(\psi_n(\omega)) = -\eta_n(\omega)$ a.s., and $\alpha_n(\psi_n(\omega)) = \alpha_n(\omega)$ a.s.\ due to $\F_{t_n}$-measurability.
Therefore, $\sup_{n\ge 1} n^a\abs{\eta_n(\omega)-\alpha_n(\omega)}<\infty$ a.s. Combining this relation with \eqref{alphan} and using the triangle
inequality, we get $M(\omega) := \sup_{n\ge 1} n^a\abs{\eta_n(\omega)}<\infty$ a.s. And since  the family  $\set{\eta_n,n\ge 1}$ is Gaussian, $M(\omega)$ has Gaussian tails thanks to Fernique's theorem \cite{fernique}. In particular, $\mex{M(\omega)^2}<\infty$. It follows that $$\sup_{n\ge1} n^{2a}\mex{\eta_n^2} = \sup_{n\ge 1} \frac{n^{2a}}{\log(n+1)}\le \mex{M(\omega)^2} <\infty,$$ which is absurd. Consequently, $\xi$ does not satisfy the assumption of Theorem~\ref{abreakthroughtheorem}.
\end{example}

\section{Discussion and applications} \label{s:finance}

\subsection{Application to finance}

On the time interval $[0,1]$ consider a fractional Black--Scholes, or simply $(B,S)$-, market with a risky asset (stock) $S$ and a non-risky asset (bond) $B$,
which solve the following equations:
\begin{align*}
&dB_t = r_t B_t dt\\
&dS_t = \mu S_t dt + \sigma S_t dB^H_t,
\end{align*}
equivalently, assuming $B_0=1$,
\begin{equation}\label{fracB-S}
\begin{aligned}
&B_t = \exp\set{\int_0^t r_s ds}\\
&S_t = S_0\exp\set{\mu t + \sigma B^H_t}.
\end{aligned}
\end{equation}
The interest rate $r$ can be random. For the technical simplicity we assume that it is absolutely bounded by a non-random constant.

Let $\mathbb F$ be the filtration generated by $B$ and $S$: $\mathcal F_t = \sigma\set{B_u,S_u, u\le t} = \sigma\set{B_u,B^H_u, u\le t}$.
Due to continuity of $B$ and $B^H$, $\mathbb F$ is left-continuous.

We remind standard notions from financial mathematics.
\begin{definition}\label{d:port}
A \emph{portfolio}, or trading strategy, is an $\mathbb F$-predictable process $\Pi = (\Pi_t)_{t\in [0,1]}
= (\pi ^0_t , \pi^1 _t )_{t\in [0,1]}$, where $\pi^0_t $ denotes the number of bonds, and $\pi^1 _t $
denotes the number of shares owned by an investor at time $t$. The \emph{value} of the portfolio
$\Pi$ at time $t$ is
\[
 V_t^\Pi = \pi^0_t B_t + \pi^1_t S_t .
\]

The portfolio is called \emph{self-financing} if
\begin{equation}
\label{sf}
dV_t^\Pi = \pi^0_t dB_t + \pi^1_t dS_t,
\end{equation}
that is, changes in the portfolio value are only due to changes in asset prices, so there is no external capital inflows
and outflows.
\end{definition}
\begin{remark}
Condition \eqref{sf} is understood in the sense that
\begin{align*}
V_t^\Pi &= V_0^\Pi + \int_0^t \left(\pi^0_u dB_u + \pi^1_u dS_u\right)\\ &= V_0^\Pi+ \int_0^t\left((\pi^0_u r_u B_u+\pi_u^1\mu S_u) du + \pi^1_u S_u dB^H_u\right),
\end{align*}i.e. for a self-financing strategy, we assume that the latter integrals are well-defined as Lebesgue and generalized
Lebesgue--Stieltjes integrals correspondingly.
\end{remark}
\begin{remark}
Thanks to the left continuity of $\mathbb F$, the property of $\mathbb F$-predictabi\-lity of process $\pi$ is equivalent to
$\mathbb F$-adaptedness.

Further, for any $\mathbb F$-adapted process $\pi^1$ such that $\int_0^1 \pi^1_u S_u dB^H_u$ is well defined and for any initial
capital $V_0$ it is possible to construct a self-financing portfolio $\Pi$ such that its risky part is $\pi^1$ and $V_0^\Pi = V_0$.
\end{remark}

Define the discounted value of a portfolio
$$
C_t^\Pi =  V_t^\Pi B_t^{-1}.
$$
It is easy to check that
$$
dC_t^\Pi = \pi_t^1 dX_t,
$$
where $X_t = S_t B_t^{-1}$ is the discounted risky asset price process. We stress once more that the integrability with respect to $X$
is understood in the following sense: we say that integral
\begin{equation}\label{intX}
\int_0^t a_s dX_s = \int_0^t \big( a_s(\mu-r_s)X_s ds + \sigma a_s X_s dB^H_s\big)
\end{equation}
exists if $a_s(\mu-r_s)X_s$ is Lebesgue integrable on $[0,t]$ and $\norm{a_{\cdot} X_{\cdot}}_{\alpha,t}<\infty$ for some $\alpha\in(1-H,1/2)$.

\begin{definition}\label{d:arb} A self-financing portfolio $\Pi $ is \emph{arbitrage} if $V_0^\Pi = 0 $,
$V_1^\Pi \ge 0 $ a.s., and $\mathsf P\left( V_1^\Pi>0\right) > 0 $.
 It is called \emph{strong arbitrage} if additionally there exists a constant $c > 0 $ such that $V_1^\Pi \ge c $ a.s.
\end{definition}

\begin{definition}
A \emph{contingent claim} on fractional $(B,S)$-market is a non-negative $\F_1$-measurable random variable $\xi$.

Contingent claim $\xi$ is called \emph{attainable}, or \emph{hedgeable}, if there exists a self-financing portfolio $\Pi$, which is
called a \emph{hedge}, or \emph{replicating portfolio}, for $\xi$, such that $V_1^\Pi = \xi$ a.s.

We will call $\xi$ \emph{weakly hedgeable} if there exists a portfolio $\Pi$ (a \emph{weak hedge}), self-financing on each interval $[0,t]$,
$t<1$, such that
$\lim_{t\to 1-} V_t^\Pi = \xi$ a.s.

The initial portfolio value $V_0^\Pi$ is called a \emph{hedging cost} or a \emph{weak hedging cost} correspondingly.
\end{definition}

Rogers \cite{Rogers97} showed that fractional $(B,S)$-market model admits arbitrage in an unconstrained case, like ours.
For more information on the arbitrage possibilities in these models, see \cite{bsv1} and references therein.

We have the following result on strong arbitrage, which we do not prove this immediately, as it will follow from the stronger result further (Theorem \ref{distrfmthm}).
\begin{theorem}
The fractional $(B,S)$-market model admits strong arbitrage.
\end{theorem}
\begin{remark}
The Ito formula from Theorem~\ref{ito} is not enough to provide a strong arbitrage. Indeed,
it gives
$$
\int_0^1 f(B_s^H) = F(B_1^H).
$$
But  $B_1^H$ is Gaussian, so it can be
arbitrarily close to $0$ with a positive probability, and since $F$ is continuous, $F(B_1^H)$ is also arbitrarily close to $0$ with
a positive probability, so strong arbitrage is impossible in this case.
\end{remark}

Now we establish an auxiliary result, similar to Lemma \ref{mainlemma}.
\begin{lemma}\label{mainlemmaX}
There exists an $\mathbb F$-adapted process $\phi=\set{\phi_t,t\in[0,1]}$ such that
for any $t<1$ integral $v_t = \int_0^t \phi_s dX_s $ is well defined in the above sense and
\item $\lim _{t\to 1-} v_t = \infty $ a.s.
\end{lemma}
\begin{proof}
Put $\phi_s = X^{-1}_s \varphi_s$, where $\varphi$ is defined in Lemma \ref{mainlemma}. Since $\phi_s(\mu-r_s)X_s
= \varphi_s(\mu-r_s)$ is almost surely bounded in $s$, and $\phi_s X_s = \varphi_s$,
we have that integral $\int_0^t \phi_s dX_s$ is well defined.
Moreover, integral $\int_0^1 \varphi_s(\mu-r_s) ds$ is finite, while $\int_0^t \varphi_s dB^H_s\to \infty$, $t\to1-$, so by \eqref{intX} we have
$v_t\to\infty$, $t\to 1-$.
\end{proof}
As a corollary, similarly to Theorem \ref{distrthm}, we have the following result.
\begin{theorem}\label{distrfmthm}
For any distribution function $F$ there is a self-financing portfolio $\Pi$ with $V_0^\Pi=0$ such that its discounted terminal
capital $C_1^\Pi$ has distribution $F$.
\end{theorem}
\begin{proof}
As in the proof of Theorem \ref{distrthm}, let $g$ be such that $g(B^H_{1/2})$ has distribution $F$, $\phi$ be as in Lemma \ref{mainlemmaX}
and $v_t = \int_{1/2}^t \phi_s dX_s$. Set
$$
\tau = \min\set{t\ge 1/2: v_t = g(B^H_{1/2})},\ \pi^1_t = \phi_s \ind{[1/2,\tau)}(t)
$$
Then it is possible to construct a self-financing portfolio $\Pi = (\pi^0,\pi^1)$ with $V_0^\Pi=0$. Clearly, $\pi^0_t=0, t\le 1/2$,
so $C_{1/2}^\Pi=0$. Further,
$$C_1^\Pi = C_{1/2}^\Pi + \int_{1/2}^1 \pi^1_s dX_s = \int_{1/2}^\tau \phi_s dX_s = g(B^H_{1/2}),$$
as  required.
\end{proof}
 If we  let $F$ to be the distribution function of some constant $A>0$ and observe that $B_1^{-1}$ is
greater than a non-random positive constant due to our assumption on $r_t$, then we derive the result about strong arbitrage.

Now we are ready to state main results of this section concerning hedging of contingent claims in the fractional $(B,S)$-market.
\begin{theorem}
Any contingent claim $\xi$ in the fractional $(B,S)$-market is weakly hedgeable. Moreover, its weak hedging cost can be any real number.
\end{theorem}
\begin{proof}
As in Theorem~\ref{thmimproper}, for any $V_0\in \R$ there is an $\mathbb F$-adapted process $\pi^1$ such that
$$
\int_0^t \pi^1_s dX_s \to \xi B_1^{-1} -V_0,\ t\to 1-.
$$
Then we can construct a self-financing portfolio $\Pi=(\pi^0,\pi^1)$ such that $V_0^\Pi = V_0$. We have
$$V_t^\Pi = B_t C_t^\Pi = B_t \left(V_0 + \int_0^t \pi_s^1 dX_s\right) \to \xi,\ t\to1-,$$
as required.
\end{proof}
\begin{theorem}
Assume that for a contingent claim $\xi$ there exists an $\mathbb F$-adapted almost surely H\"older continuous process $\set{z_t,t\in[0,1]}$
with $z_1 = \xi$. Then $\xi$ is hedgeable and its hedging cost can be any real number.
\end{theorem}
\begin{proof}
As in the previous theorem, it is for enough to construct an $\mathbb F$-adapted process $\pi^1$ such that $\int_0^1 \pi^1_t dX_t$ is
well defined and
$$
V_0 + \int_0^1 \pi^1_t dX_t = \xi B_1^{-1}.
$$
To that end we slightly modify the construction from Theorem~\ref{abreakthroughtheorem}. Namely, we first take some $\gamma >(1-\alpha-H+a)^{-1}>1$ and put $\Delta_n = n^{-\gamma}/\zeta(\gamma)$,  $t_n = \sum_{k=1}^{n} \Delta_k$, $\Delta B^H_n = B^H_{t_{n}} - B^H_{t_{n-1}}$, $\xi_n = z_{t_n}B_{t_n}^{-1}$.
We will also use the notation $B^H_{x,y} = B^H_x - B^H_y$.

Then we construct process $\pi^1$ recursively on intervals $[t_n,t_{n+1})$, starting by setting $\pi^1_t=0$ on $[0,t_1)$. Then we take some $\kappa\in(\gamma(H-a),\gamma(1-\alpha)-1)$.

Denote $v_t = \int_0^t \pi^1_s dX_s$.  If $\psi$ is constructed up to $t_n$, we define $$\tau_n = \min\set{t\ge t_n: n^{\kappa}\abs{\int_{t_n}^t
(\mu -r_s)\sign B^H_{s,t_n}ds
+\sigma B^H_{t,t_n}} = \abs{v_{t_n}-\xi_n}}\wedge t_{n+1}$$ and set
$$\pi^1_t = n^\kappa \sign B^H_{t ,t_n} \sign(\xi_n-v_{t_n})\ind{t< \tau_n}X_t^{-1}.$$ for $t\in[t_n,t_{n+1})$.

The rest of proof is the same as in Theorem~\ref{abreakthroughtheorem}, so we do not repeat it fully, making only important remarks.
The Step 2 of proof will be still true, since $\int_{t_n}^t (\mu -r_s)\sign B^H_{s,t_n}ds$ is of order $(t-t_n)$
which is negligible compared to the quantities involved in this step. In the Step 3, we should not consider
expectations immediately. Instead, we note that $X^{-1}$ is almost surely H\"older continuous of any order less than $H$
and  estimate for $t,s\in[0,\tau_n) $
\begin{gather*}
\abs{\psi_s - \psi_t} \le C(\omega) n^\kappa\left(\ind{B^H_{t,t_n}B^H_{s,t_n}<0} + \abs{X^{-1}_t - X^{-1}_s}\right)\\
\le C(\omega) n^\kappa\left(\ind{B^H_{t,t_n}B^H_{s,t_n}<0} + (t - s)^d\right),
\end{gather*}
where $d\in(\alpha,H)$ is such that $\kappa - \gamma(1-d+H-\alpha)<-1$.  For other $t,s$ we write simply
$$
\abs{\psi_s - \psi_t}\le C(\omega).
$$
Then we take expectation only of the term involving $\ind{B^H_{t,t_n}B^H_{s,t_n}<0}$ (without the random constant) and we know that it is finite.
The rest of terms are easily checked to be finite exactly as in Theorem~\ref{abreakthroughtheorem}, and
multiplication by a random  constant cannot make things infinite. \end{proof}

\subsection{Zero integral}

Assume that process $g$ is adapted to $\mathbb F$, and
the integral $\int _0^1 g_sdB^H_s$ is well defined.
We are interested in the following question:
\begin{itemize}
 \item If $\int _0^1 g_u dB^H_u = 0 $ a.s., is it true that $g=0 $
almost everywhere with respect to $\mathbb P \otimes \lambda $?
\end{itemize}
Recall the following fact for standard Brownian motion $W$. Assume that
$
 \int _0^1 \mex{h_s^2} ds < \infty$. Then we have the following equivalence from
the It{\^o}-isometry
$$
\int _0^1 h_s dW_s =0 \Leftrightarrow h=0 \quad
\mathbb P \otimes \lambda \mbox{-a.e.}
$$
If the integrability assumption is replaced by $\int _0^1 h_s^2 ds < \infty$, then the conclusion is false:
it is proved in \cite{dudley} that one can construct an adapted process $h$ such that $\int_0^{1/2} h_s dW_s = 1$ and
$\int_{1/2}^1 h_s dW_s = -1$, so $h$ cannot equal zero identically.

Similarly, in the fractional Brownian framework, thanks to Theorem~\ref{distrthm} we can construct $g$ adapted to the filtration generated by the fractional Brownian motion $B^H$ such that $\int_0^{1/2} g_s dB^H_s = 1$ and
$\int_{1/2}^1 g_s dB^H_s = -1$. This gives a negative  answer to the question we are interested in.

Nevertheless, in some special cases we can conclude that the integrand $g$ is zero when the
integral $\int_0^1 g_s dB^H_s = 0$. One can show this for integrands with finite fractional chaotic expansion and for
simple predictable integrands.

First consider $g$ with a finite fractional Wiener--It\^o expansion. We give only
brief explanation here, the details can be found e.g.\ in \cite[Chapter 3]{biagini}.
Let $\phi(t,s)  = H(2H-1) \abs{t-s}^{2H-2}$ and define a scalar product of functions $f,g:[0,1]\to \R$
\begin{equation*}
\pair{f,g}_H = \int_{[0,1]^2} \phi(t,s)f(t)g(s)dt\,ds
\end{equation*}
and the corresponding norm $\norm{f}_H = \pair{f,f}^{1/2}_H$. The space $L^2_H([0,1])$ is the space of functions $f:[0,1]\to\R$
such that $\norm{f}_H<\infty$. Consider the $n$th symmetric tensor power of $L^2_H([0,1])$: $\wLH([0,1]^n) = L^2_H([0,1])^{\wotimes n}$.
It inherits a Hilbertian structure from $L_H^2([0,1])$: for $f,g\in \wLH([0,1]^n)$
\begin{multline*}
\pair{f,g}_{H}^2= \pair{f,g}_{\wLH([0,1]^n)}^2 = \\\int_{[0,1]^{2n}} f(t_1,\dots,t_n)g(s_1,\dots,s_n)\phi(t_1,s_1)\dots \phi(t_n,s_n) dt_1\dots dt_n\,ds_1\dots ds_n.
\end{multline*}
For a function $f\in \wLH([0,1]^n)$ it is possible to define the iterated stochastic integral
$$
I_n(f) = \int_{[0,1]^n} f(t_1,\dots,t_n) dB^H_{t_1}\wick\dots\wick dB^H_{t_n}
$$
(we use the symbol $\wick$ here to emphasize that this integral differs from the \emph{pathwise} iterated integral.)
Now let $g=\set{g_t,t\in[0,1]}$ have finite fractional chaos expansion of the form
$$
g_t=\sum_{k=0}^n I_{k}(f_{k}(\cdot,t))
$$
with $f_k(\cdot,t)\in \wLH([0,1]^k)$, $k\le n$. Thanks to finiteness of expansion, the process $g_t$ has a stochastic derivative
and
$$
D_s g_t = \sum_{k= 1}^{n} k I_{k-1}(f_{k}(\cdot,s,t)).
$$
Assume that  $f_k \in \wLH([0,1]^{k+1})$, $k\le n$. Then the process $g$ belongs to the domain of the divergence integral $\delta$ and
$$
\delta(g) =  \sum_{k= 1}^{n} I_{k+1}(\widetilde f_k),
$$
where $\widetilde f_k$ is the symmetrization of $f$.
Also assume that
$$
T(g) = \int_{[0,1]^2} \abs{D_s g_t} \phi(t,s) ds\, dt<\infty
$$
a.s. A checkable sufficient condition for this is
\begin{gather*}
\mex{T(g)^2} = \sum_{k=1}^n k! \int_{[0,1]^4}\pair{f_k(\cdot,s,t),f_k(\cdot,u,v)}_{H}\phi(t,s)\phi(u,v)ds\,dt\,du\,dv<\infty.
\end{gather*}
Then it is known (see \cite[Proposition 4.1]{biagini}) that there is a relation between pathwise and divergence integrals:
\begin{equation*}
%\label{forward-skor-relation}
\int_0^1 g_t dB^H_t =  \delta(g) + \int_{[0,1]^2}  D_s g_t \phi(s,t)ds\,dt.
\end{equation*}
If $\int_0^1 g_t dB^H_t=0$, then
\begin{equation}\label{rel1}
\delta(u) = - \int_{[0,1]^2} D_s u_t \phi(s,t)ds\, dt.
\end{equation}
The chaotic expansion of the right-hand side is
$$
- \int_{[0,1]^2} D_s u_t \phi(s,t)ds\, dt = \sum_{k=0}^{n-1} (k+1) I_{k}(h_k),
$$
where
$$h_k(t_1,\dots,t_{k}) = -\int_{[0,1]^2} f_{k+1}(t_1,\dots,t_{k},s,t)\phi(t,s)ds\,dt.$$
Since the Wiener--It\^o expansion is unique, the corresponding coefficients of the left-hand and the right-hand sides
of equality \eqref{rel1} are the same. In particular, taking $(n+1)$th terms of expansion, we get $I_{n+1}(\widetilde f_n) =0$ a.s.,
whence $\widetilde f_n=0$ $\lambda$-a.e., consequently, $f_n=0$ $\lambda$-a.e. Using a backward induction, we get
that $f_k = 0$ $\lambda$-a.e.\ for all $k\le n$, concluding that $g = 0$ $P \otimes \lambda$-a.e.

\smallskip

Assume now that $g $ is a simple predictable process of the form
\[
 g = \sum _{k=1} ^m \alpha _k\mathbf 1_{[t_{k-1}, t_k)} ,
\]
where $0 = t_0 < t_1 < \cdots < t_m = 1$, $\alpha_k \in \mathcal{F}_{t_{k-1}}$.
If
\[
\int_0^1 g_t dB^H_t  = \sum _{k=1}^m  \alpha _k \Delta B^H_{t_k} = 0,
\]
then it was proved in \cite[Theorem 2.5.1]{m} that $\alpha _k = 0  $ a.s. A proof uses
the representation of fractional Brownian motion with respect to standard Brownian motion on the finite interval considered in  \cite{nvv}.

\bibliographystyle{elsarticle-num}
\bibliography{integrfbm}
\end{document}